\documentclass[cupthm,crop,info]{CUP-JNL-ETS}%

\usepackage{graphicx}
\usepackage{multicol,multirow}
\usepackage{amsmath,amssymb,amsfonts}
\usepackage{mathrsfs}
\usepackage{rotating}
\usepackage{appendix}
\usepackage[numbers]{natbib}
\usepackage{amsthm}
\usepackage{lipsum}

\theoremstyle{cupplain}
\newtheorem{theorem}{Theorem}[section]
\newtheorem{lemma}[theorem]{Lemma}
\newtheorem{corollary}[theorem]{Corollary}
\theoremstyle{cupdefinition}
\newtheorem{definition}{Definition}[section]
\theoremstyle{cupremark}
\newtheorem{remark}[theorem]{Remark}
\newtheorem{example}[theorem]{Example}
\theoremstyle{cupproof}
\newtheorem{proof}{Proof}

\numberwithin{equation}{section}

\jyear{2020}
\jdoi{10.1017/etds.2020.xx}

\begin{document}

\newcommand{\PP}{{\mathbb P}}

\newcommand{\R}{{\mathbb R}}
\newcommand{\F}{{\mathbb F}}
\newcommand{\N}{{\mathbb N}}
\newcommand{\Z}{{\mathbb Z}}
\newcommand{\E}{{\mathbb E}}
\newcommand{\T}{{\mathbb T}}
\newcommand{\D}{{\cal D}}
\newcommand{\A}{{\cal A}}
\newcommand{\pder}[2]{\frac{\partial #1}{\partial #2}}

\begin{Frontmatter}

\title{A Kakutani-Rokhlin decomposition for conditionally ergodic process  in the measure-free setting of vector lattices}

\author{\gname{Youssef} \sname{Azouzi}}
\address{\orgdiv{LATAO Laboratory, Faculty of Mathematical, Physical and Natural Sciences of Tunis}, \orgname{Tunis-El Manar University}, \orgaddress{\city{Tunis}, \postcode{2092 El Manar}, \state{State},  \country{Tunisia}}\\ 
(\email{josefazouzi@gmail.com},\email{youssef.azouzi@ipest.rnu.tn})}
\author{\gname{Marwa} \sname{Masmoudi}}
\address{\orgdiv{LATAO Laboratory, Faculty of Mathematical, Physical and Natural Sciences of Tunis}, \orgname{Tunis-El Manar University}, \orgaddress{\city{Tunis}, \postcode{2092 El Manar}, \state{State},  \country{Tunisia}}\\ 
(\email{marwa_masmoudi@hotmail.com})}
\author{\gname{Bruce} \sname{Watson}}
\address{\orgdiv{School of Mathematics, CoE-MaSS \& NITheCS}, \orgname{University of the Witwatersrand}, \orgaddress{\city{Johannesburg}, \postcode{2050}, \state{Gauteng},  \country{South Africa}}\\
(\email{b.alastair.watson@gmail.com},\email{bruce.watson@wits.ac.za})\\
ORCID: 0000-0003-2403-1752
}

\Received{\textup{00} November \textup{2019}}
\Accepted[ and accepted in revised form]{\textup{00} February \textup{2020}}

\maketitle

\authormark{Y. Azouzi, M. Masmoudi and B.A. Watson}
\titlemark{Kakutani-Rokhlin decomposition in vector lattices}

\begin{abstract}
Recently the Kac formula for the conditional expectation of the first recurrence time of a conditionally ergodic conditional expectation preserving system was established in the measure free setting of vector lattices (Riesz spaces).  We now give a formulation of the Kakutani-Rokhlin decomposition for conditionally ergodic systems in terms of components of weak order units in a vector lattice.  
In addition, we prove that every aperiodic conditional expectation preserving system can be approximated by a periodic system.
\end{abstract}



\end{Frontmatter}

\section{Introduction} \label{s: introduction}

The ergodic theorems of Birkhoff, Hopf,  von Neumann, Wiener and Yoshida were generalized to the measure free setting of vector lattices (Riesz spaces) in 2007,  see \cite{kuo2007ergodic}.
The Poincaré recurrence theorem and Kac's formula in vector lattices were published in 2023, see \cite{azouzi2022kac}.
In none of the above were the concept of Rokhlin Towers/Kakutani-Rokhlin decomposition used and up to the present the concept of  Rokhlin Towers/Kakutani-Rokhlin decomposition had not been generalized to the vector lattice setting.
In this paper, we present a Kakutani-Rokhlin decomposition for dynamical systems defined by
iterates of a Riesz homomorphism acting on a vector lattice (Riesz space).
This takes the work of \cite{ eisner2015operator,   petersen1989ergodic, rokhlin1967lectures} and others, out of the realm of
metric, topological and measure spaces.
We note here the contrast between the development in vector lattices and that of \cite{bochi}, where the ergodic theorems were derived as a consequence of the Rokhlin Towers/Kakutani-Rokhlin decomposition.  
Our generalization to vector lattices is with respect to a discrete-time process, but generalizes the underlying space of the process.  Other generalizations, see for example \cite{lindenstrauss}, 
have kept the underlying space as a measure space but have extended the time-index space to amenable groups.
We refer the reader to  \cite{kra2018commentary, weiss1989work} for  some applications and other generalizations of the Rokhlin Towers/Kakutani-Rokhlin decomposition.

The extension given here, when applied back to probabilistic systems gives the existence 
of a Kakutani-Rokhlin decomposition (also known in this context as Rokhlin towers) for conditionally ergodic processes.  A consequence of this is that every conditional expectation preserving system that is aperiodic admits a Kakutani-Rokhlin decomposition.  Examples are given in each stage of our development to show that the given result cannot be improved without additional assumptions.

 In the probability setting, the Kakutani-Rokhlin lemma gives that, for each $n\in \N$ and $\epsilon >0$ each a.e. bijective ergodic aperiodic measure preserving transformation, $\tau$,  on the probability space $(\Omega,{\cal A},\mu)$, there is a set $B\in{\cal A}$ so that 
 $\tau^{-j}(B), j=0,\dots,n-1,$ are disjoint and
  $\displaystyle{\mu\left(\Omega\setminus \cup_{j=1}^n \tau^{1-j}(B)\right)<\epsilon}$.  We refer the reader to \cite[Section 3.3]{rokhlin1967lectures} for the specific result and to \cite{kornfeld} for a survey of research around such decompositions.
The generalization of the Kakutani-Rokhlin lemma to the topology-free, metric-free, measure-free setting of Riesz spaces (vector lattices) is given in Theorem \ref{theo ergodic}.  

We begin by giving an $\epsilon$-free version of the Kakutani-Rokhlin lemma in the Riesz space setting, Theorem \ref{kakutani1},  see \cite{weiss1989work} for the measure space version.    
This version applies to conditionally ergodic systems on Riesz spaces (in fact each conditional expectation preserving system of a Riesz space gives rise to a conditionally ergodic system) and does not require aperiodicity.  This then forms the foundation of Theorem \ref{theo ergodic}, where aperiodicity is essential.

The remaining foundational aspects of ergodic theory in Riesz spaces needed for the current work can be found in  \cite{amor2023characterisation,homann2020ergodicity} and \cite{kuo2005conditional} for the general theory of conditional expectation operators in Riesz spaces.
It should be noted that many other stochastic processes have been studied in the Riesz space (vector lattice) framework, for example discrete \cite{klw-disc} and continuous \cite{gl-cts-1, gl-cts-2, stoica} time martingale processes as well as mixing processes \cite{amor2023characterisation}.

In Section 2, we recall the basics of conditional expectation preserving systems, ergodic processes,  Poincaré's recurrence theorem and Kac's formula in Riesz spaces.
In Section 3, we  give a Riesz space version of the $\epsilon$-free Kakutani-Rokhlin type decompositions.
In Section 4 we introduce aperiodicity in Riesz spaces and use this concept together with the Kac formula and  the $\epsilon$-free Kakutani-Rokhlin type decomposition to give an $\epsilon$-bound version of the Kakutani-Rokhlin decomposition in Riesz spaces.
We conclude in Section 5 with an application of the Kakutani-Rokhlin Theorem for aperiodic processes to show that every conditionally ergodic process which can be decomposed into aperiodic processes can be approximated by a periodic processes.

\section{Preliminaries} \label{preliminary section}

For Riesz space theory and associated terminology,  we refer readers to \cite{zaanen2012introduction}.  The background material on ergodic theory can be found in \cite{eisner2015operator,petersen1989ergodic}.  Our current work builds on \cite{azouzi2022kac},  in which many of the foundational results can be found.
The concept of a conditional expectation operator on a Riesz space is fundamental to the material presented here and hence we quote its definition from \cite{kuo2005conditional}.

\begin{definition} 
Let $E$ be an Archimedean Riesz space with weak order unit. A positive order continuous projection $T \colon E \rightarrow E$ which maps weak order units to weak order units is called a conditional expectation if the range of $T$, $R(T)$,  is a Dedekind complete Riesz subspace of $E$.
\end{definition}

Throughout this work we will assume that the conditional expectation operator $T$ is strictly positive, that is, if $f\in E_+$ and $Tf=0$ then $f=0$.

The Riesz space analogue of a measure preserving system was introduced in \cite{homann2021koopman} as a conditional expectation preserving system, see below. The concept was first used and studied in \cite{kuo2007ergodic}, but not given a name there.  

\begin{definition} \label{system definition}
The 4-tuple, $(E,T,S,e)$, is called a conditional expectation preserving system (CEPS) if $E$ is a Dedekind complete Riesz space, $e$ is a weak order unit of $E$, $T$ is a conditional expectation operator on $E$ with $Te=e$, $S$ is an order continuous Riesz homomorphism on $E$ with $Se=e$ and $TSf=Tf$, for all $f \in E$.
\end{definition}

\begin{remark}
If $(E,T,S,e)$ is a conditional expectation preserving system,
then $$TS^jf=Tf$$ for all $j \in \N_0$ and $f \in E$.

We also note that if $S$ is a Riesz homomorphism with $TS=T$ where $T$ is a strictly positive conditional expectation operator on a Dedekind complete Riesz space $E$, then $S$ is order continuous, see \cite{KSvG} for a more general study of  order continuity Riesz homomorphism. To see this, we let $f_\alpha$ be a downwards directed net in $E^+$ with $f_\alpha \downarrow 0$, then $Sf_\alpha$ is downwards directed with $Sf_\alpha\downarrow h$ for some $h\in E^+$. However $TS=T$ so, as $T$ is order continuous, 
$0\leftarrow Tf_\alpha=T(Sf_\alpha)\to Th$. The strict positivity of $T$ now gives $h=0$. Hence $Sf_\alpha\to 0$, making it order continuous.  
\end{remark}

We recall, from \cite{kuo2005conditional}, the concept of $T$-universal completeness, the $T$-universal completion and, from \cite{kuo2017mixing}, the $R(T)$-module structure of $L^1(T)$,  see also \cite{AT}.

\begin{definition}
If $T$ is a strictly positive conditional expectation operator on a Dedekind complete Riesz 
space, $E$ with weak order unit $e=Te$, then the  natural domain of $T$ is 
$$\mbox{dom}(T):=\{f\in E^u_+|\exists  \mbox{ net } f_\alpha\uparrow f \mbox{ in } E^u,
(f_\alpha)\subset E_+, Tf_\alpha \mbox{ bounded in } E^u\},$$
where $E^u$ denotes the universal completion of $E$.
We define
$$L^1(T):=\mbox{dom}(T)-\mbox{dom}(T)=\{f-g|f,g\in\mbox{dom}(T)\}$$
and say that $E$ is $T$-universally complete if $E=L^1(T)$. 
\end{definition}

From the above definition, $E$ is $T$-universally complete if, and only
 if, for each upwards directed net $(f_{\alpha})_{\alpha \in \Lambda}$ in $E_+$ such that $(Tf_{\alpha})_{\alpha \in \Lambda}$ is order bounded in $E^{u}$, we have that $(f_{\alpha})_{\alpha \in \Lambda}$ is order convergent in $E$. 
 
 $E^u$ has an $f$-algebra structure which can be chosen so that $e$ is the multiplicative identity.
For $T$ acting on $E=L^1(T)$,  $R(T)$ is a universally complete and thus an $f$-algebra, and, further, $L^1(T)$ is an $R(T)$-module.
From \cite[Theorem 5.3]{kuo2005conditional}, $T$ is an averaging operator,  which means that if $f\in R(T)$ 
and $g\in E$ then $T(fg)=fT(g)$.

From \cite{kuo2007ergodic},  for each $f\in L^1(T)$, the Ces\`aro mean
\begin{equation}\label{cesaro}
L_Sf:=\lim_{n\to\infty}\frac{1}{n}\sum_{k=0}^{n-1}S^kf,
\end{equation}
converges in order, in $L^1(T)$, for each Riesz homomorphism $S$ on $E=L^1(T)$ with $TS=T$ and $Se=e$.
We denote the invariant set of the Riesz homomorphism, $S$, by 
$${\cal I}_S:=\{f\in L^1(T) : Sf=f\}.$$

We say that $p\in E_+$ is a component of $q\in E_+$ if $p\le q$ and $(q-p)\wedge p=0$.
We denote the set of components of $q$ by $C_q$.

The conditional expectation preserving system $(E=L^1(T),T,S,e)$ is said to be conditionally ergodic if $L_S=T$ which is equivalent to ${\cal I}_S=R(T)$, see \cite{homann2020ergodicity},  in which case $ST=T$ and hence $S^jTf=Tf$ for all $j \in \N_0$ and $f \in E$.

\begin{lemma}\label{baw-12-com}
If $(E,T,S,e)$ is a conditional expectation preserving system and $T$ is strictly positive then $Sg=g$ for all $g\in R(T)$. In the case of $E$ being an $R(T)$ module this invariance gives that $S(gf)=gSf$ for all $g\in R(T)$ and $f\in E$.
\end{lemma}

\begin{proof}
Due to the order continuity of $S$ and the order density of the linear combinations of components of $e$ in $E$, it suffices to prove the result for $g\in C_e\cap R(T)$ and $f\in C_e$. 

For $g\in C_e\cap R(T)$ we have that 
$$g=Tg=TSg.$$
The averaging property of conditional expectations operators in terms of band projections  gives 
that $P_{TSg}\ge P_{Sg}$ where these are respectively the band projections generated by 
$TSg$ and $Sg$,  see \cite[Corollary 2.3]{klw-f} and \cite[Lemma 2.3]{klw-convergence}.   
Here $Sg$ is a component of $e$ so $P_{Sg}e=Sg$.
Further as $g=TSg$ which is a component of $e$ we have $P_{TSg}e=g$. 
Thus $g\ge Sg$.  As $T$ is strictly positive and $S$ is a Riesz homomorphism by \cite[Note 2.3]{azouzi2022kac} we have $Sg=g$.

For the second result, if $f\in C_e$ then $fg=f\wedge g$ so
$$S(gf)=S(g)\wedge S(f)=g\wedge S(f) =gSf$$
since $Sf$ is also a component of $e$.
\end{proof}

In  \cite[Lemma 3.1]{azouzi2022kac} an equivalent formulation for the definition of recurrence in \cite[Definition 1.4]{azouzi2022kac} was proved. For convenience here we will take this equivalent statement as a our definition of recurrence below.

\begin{definition}[Recurrence] \label{recurrence definition}
Let $(E,T,S,e)$ be a conditional expectation preserving system with $S$ bijective, then  $p\in C_q$ is recurrent with respect to $q\in C_ e$ if $$p\le\bigvee_{n\in\N} S^{-n} q.$$
\end{definition}

The following Riesz space generalization of the Poincaré recurrence theorem was proved in \cite[Theorem 3.2]{azouzi2022kac}.

For brevity of notation, we define the supremum over an empty family of components of $e$ to be zero, i.e.
$$\bigvee\limits_{j=1}^{0}p_j:=0,$$
for $(p_j)\subset C_e$.

\begin{theorem}[Poincaré]\label{thm-poincare}
Let $(E,T,S,e)$ be a conditional expectation preserving system with $T$ strictly positive and  $S$ surjective, then each  $p\in C_q$ is recurrent with respect to $q$ for each $q\in C_e$. 	
\end{theorem}

For $k \in \N$, let $$q(p,k):=p \wedge (S^{-k}p) \wedge (e-\bigvee\limits_{j=1}^{k-1}S^{-j}p).$$

Here $q(p,k)$ is the maximal component of $p$ recurrent at exactly $k$ iterates of $S$ and 
$$q(p,k)\wedge q(p,m)=0,\quad \mbox{for} \quad k\ne m,\quad  k,m\in\N.$$

Writing Theorem \ref{thm-poincare} in terms of $q(p,k),  k\in\N$ we obtain the next corollary.

\begin{corollary}\label{decomposition p}
	Let $(E,T,S,e)$ be a conditional expectation preserving system with $T$ strictly positive and  $S$ surjective, then	 for each component $p$ of $e$ we have 
	$$p=\sum\limits_{k=1}^\infty q(p,k).$$
	Here this summation is order convergent in $E$.
\end{corollary} 

From the definition of $q(p,k)$ we have that  
$$S^k q(p,k) \le p,$$ 
for all  $k \in \N$.

\begin{lemma}\label{q-disjoint}
Let $(E,T,S,e)$ be a conditional expectation preserving system with $T$ strictly positive and  $S$ surjective, then
for all $m,n \in \N$ with $0 \le i \le m-1, \; 0\le j \le n-1$ and $(i,m) \ne (j,n)$ we have 
\begin{equation}\label{disjointness-1}
S^i q(p,m) \wedge S^j q(p,n)=0.
\end{equation}
\end{lemma}

\begin{proof}
Let $i,j,m,n$ be as above.

Case I: If $m\le n$ and $n-1\ge j> i\ge 0$, then as $S^j$ is a Riesz homomorphism,
\begin{eqnarray*}
S^i q(p,m) \wedge S^j q(p,n)=S^j(S^{i-j}q(p,m) \wedge q(p,n)).
\end{eqnarray*}
Here
\begin{eqnarray*}
 S^{i-j}q(p,m) \wedge q(p,n)\le S^{i-j}p \wedge  \left(e-\bigvee\limits_{k=1}^{n-1}S^{-k}p\right)=0 
\end{eqnarray*}
since $i-j\in\{-k|k=1,\dots,n-1\}$.

Case II: If $m<n$ and $i=j$, then
\begin{eqnarray*}
S^i q(p,m) \wedge S^j q(p,n)=S^i(q(p,m) \wedge q(p,n))=0
\end{eqnarray*}
as $m\ne n$ and $S^i$ is a Riesz homomorphism.

Case III: If $m<n$ and $m-1\ge i> j\ge 0$, then 
\begin{eqnarray*}
S^i q(p,m) \wedge S^j q(p,n)=S^i(q(p,m) \wedge S^{j-i}q(p,n)).
\end{eqnarray*}
Here
\begin{eqnarray*}
 q(p,m) \wedge S^{j-i}q(p,n)\le  \left(e-\bigvee\limits_{k=1}^{m-1}S^{-k}p\right)\wedge S^{j-i}p=0 
\end{eqnarray*}
since $j-i\in\{-k|k=1,\dots,m-1\}$.
\end{proof}

Rewriting the expression for the first recurrence time  for $p$ a component of $e$, $n(p)$, from \cite{azouzi2022kac}, in terms of $q(p,k)$, we get
$$n(p)=\sum\limits_{k=1}^\infty k q(p,k).$$

The conditional Kac formula of \cite{azouzi2022kac} gives the conditional expectation of $n(p)$, as follows.

\begin{theorem}[Kac]\label{kac-thm}
 Let $(E,T,S,e)$ be a conditionally ergodic conditional expectation preserving system where $T$ is strictly positive, $E$ is $T$-universally complete and $S$ is surjective. 
For each $p$ a component of $e$ we have that 
 $$Tn(p)=P_{Tp}e$$
 where $P_{Tp}$ is the band projection onto the band in $E$ generated by $Tp$.
\end{theorem}

 
 \section{Kakutani-Rokhlin Lemma - $\epsilon$-free}
 
We recall  an $\epsilon$-free version  of the Kakutani-Rokhlin decomposition for ergodic 
measure-preserving systems from \cite[Theorem 2]{lehrer-weiss} and \cite[Theorem 6.24]{eisner2015operator}. We note here that these references state the bound $1-n\mu(A)$, however their proofs yield the better bound given below.

\begin{theorem}
	Let  $(\Omega,{\cal B},\mu,\tau)$ be an ergodic measure preserving system, let $A \in  {\cal B}$ with $\mu(A)>0$ and $n \in \N$. Then there is a set $B \in  {\cal B}$ such that $B, \tau^{-1} B, ..., \tau^{1-n}B$ are pairwise disjoint and $$\mu\left(\bigcup\limits_{i=0}^{n-1}\tau^{-i} B\right) \ge 1-(n-1)\mu(A).$$ 
\end{theorem}

We now give a conditional Riesz space version of the previous result. If, in the following result, $p$ is taken as the characteristic function, $\chi_A$,  with $A$ of the above result,  and $T$ is the expectation with respect a probability measure $\mu$, then the below result yields immediately the above result. However, if $T$ is a conditional expectation, then the below result yields the above but with $\mu$ being the conditional probability induced by $T$.

\begin{theorem}[Kakutani-Rokhlin lemma] \label{kakutani1}
Let $(E,T,S,e)$ be a conditionally ergodic conditional expectation preserving system with $S$ surjective. 
Let $n \in \N$ and $p$ be a component of $e$,
then there exists a component  $q$ of $P_{Tp}e$ such that $q, Sq, \dots, S^{n-1}q$ are pairwise disjoint and \begin{equation}\label{april-thm}
T\left(\bigvee\limits_{j=0}^{n-1}S^j q\right) \ge \left(P_{Tp}e - (n-1) Tp\right)^+.
\end{equation}
\end{theorem}

\begin{proof}
By Corollary \ref{decomposition p}, $p$ can be decomposed into a sum of disjoint components as follows
 $$p=\sum\limits_{i=1}^\infty q(p,i).$$
Let $$R_k=\sum\limits_{i=k}^\infty q(p,i)=\bigvee\limits_{i=k}^\infty q(p,i),$$
then $R_k$ is the maximal component of $p$ with no component recurrent in under $k$ steps.

For fixed $n\in \N$ and $k,j\in\N_0$ with $k\ne j$, from (\ref{disjointness-1}), we have 
\begin{equation}\label{april-13-1}
S^{nj}R_{n(j+1)}\wedge S^{nk}R_{n(k+1)}
=\bigvee\limits_{i\ge n(j+1)} \bigvee\limits_{r\ge n(k+1)} \left(S^{nj}q(p,i)\wedge S^{nk}q(p,r)\right)=0,
\end{equation}
since $nj<i$, $nk<r$ and $nj\ne nk$.
Let
\begin{equation}\label{eq-q}
 q:=\sum\limits_{j=0}^{\infty} S^{nj}R_{n(j+1)}=\bigvee\limits_{j=0}^{\infty} S^{nj}R_{n(j+1)}= \sum\limits_{j=0}^{\infty}\sum\limits_{i \ge n(j+1)}S^{nj} q(p,i).
 \end{equation}
Here $q$ is a component of $e$. 

We now show that $S^i q \wedge S^j q=0$, for all $0\le i<j \le n-1$. 
For this it suffices to prove that $q \wedge S^k q=0$, for all $1 \le k \le n-1$.
If $j,m\in\N_0$, with $i \ge n(j+1)$ and $r \ge n(m+1)$ then 
$nj\ne nm+k$, $i>nj$ and $r>nm+k$, so, by (\ref{disjointness-1}), 
\begin{equation}\label{x-k}
q\wedge S^k q= \sum\limits_{j,m=0}^{\infty}\sum\limits_{i \ge n(j+1)}
\sum\limits_{r \ge n(m+1)}S^{nj} q(p,i)\wedge  S^{nm+k} q(p,r)=0.
\end{equation}
Thus  $q, Sq,\dots ,S^{n-1}q$ are disjoint.

We now proceed to the proof of (\ref{april-thm}).
From the definition of $q$ in (\ref{eq-q}) we have 
$$\bigvee\limits_{i=0}^{n-1}S^i q=\sum\limits_{i=0}^{n-1}S^i q =\sum\limits_{i=0}^{n-1}\sum\limits_{j=0}^{\infty}\sum\limits_{k \ge n(j+1)}S^{nj+i} q(p,k).$$
Applying $T$ to the above equation and using $TS^i=T, i\in\N_0,$ along with the order continuity of $T$, we have 
\begin{align*}
T\left(\bigvee\limits_{k=0}^{n-1}S^k q\right)
&=\sum\limits_{k=0}^{n-1}\sum\limits_{j=0}^{\infty}\sum\limits_{i= n(j+1)}^\infty T q(p,i)\\
&=\sum\limits_{j=0}^{\infty}\sum\limits_{i=n(j+1)}^\infty nTq(p,i)\\
&=\sum\limits_{i=0}^\infty n \left[\frac{i}{n}\right]Tq(p,i).
\end{align*}

On the other hand,  by the Riesz space version of the Kac Theorem, i.e. Theorem \ref{kac-thm},  we have
 $$P_{Tp}e=Tn(p)=\sum\limits_{i=1}^\infty i Tq(p,i).$$
Therefore,
$$P_{Tp}e-T\left(\bigvee\limits_{i=0}^{n-1}S^i q\right) 
=\sum\limits_{i=1}^\infty n\left(\frac{i}{n}-\left[\frac{i}{n}\right]\right) Tq(p,i) \le \sum\limits_{i=1}^\infty (n-1) Tq(p,i) =(n-1)Tp,$$
concluding the proof of (\ref{april-thm}). \end{proof}

\begin{example}
We now  give an example of where Theorem \ref{kakutani1} cannot be improved to the $\epsilon$ approximation of Theorem \ref{theo ergodic}.  Consider the Riesz space $E=\R\times\R$ with componentwise ordering and weak order unit $e=(1,1)$ and order continuous Riesz homomorphism $S(x,y)=(y,x)$. We take as our conditional expectation 
$T(x,y)=\frac{x+y}{2}(1,1)$.  It is easily verified that  $(E,T,S,e)$ is a conditionally ergodic conditional expectation preserving system.   Taking $p=(1,0)$ in Theorem  \ref{kakutani1} we have that $Tp=\frac{1}{2}(1,1)$ giving
$$(P_{Tp}e-(n-1)Tp)^+=\left((1,1)-\frac{n-1}{2}(1,1)\right)^+
=\left\{\begin{array}{ll}(1,1),&n=1\\ \frac{1}{2}(1,1),&n=2\\ (0,0),& n\ge 3\end{array}\right..$$
The required components of $P_{Tp}e=e$ for the respective values of $n$ are $q_1=(1,1)$, $q_2=(0,1)$ and $q=(0,0)$ for $n\ge 3$. Then $q_n,\dots S^{n-1}q_n$ are disjoint and
$$T\left(\bigvee\limits_{j=0}^{n-1}S^j q_n\right)
=\left\{\begin{array}{ll}(1,1),&n=1\\
(1,1),&n=2\\
(0,0),&n\ge 3\end{array}\right..$$
For this example the $\epsilon$ bound of Theorem \ref{theo ergodic} fails for $0<\epsilon<\frac{1}{2}$ and $n=2$.
\end{example}

\begin{example}\label{remark-2025-1}
As in \cite[Section 5]{azouzi2022kac}, let $(\Omega,{\cal A},\mu)$ be a probability space, where $\mu$ is a complete measure.  Let $\Sigma$ be a sub-$\sigma$-algebra of ${\cal A}$. 
As the Riesz space $E$ we take the space of a.e. equivalence classes of measurable functions $f:\Omega\to \R$ for which the sequence $(\E[\min(|f(x)|, {\bf n})|\Sigma])_{n\in\N}$, is bounded above by an a.e. finite valued measurable function. 
Here ${\bf n}$ is the (equivalence class of the) function with value $n$ a.e.
For $f\in E$ with $f\ge 0$ we define
$$Tf=\lim_{n\to\infty} \E[\min(f(x), {\bf n})|\Sigma]$$
in the sense of a.e. pointwise limits. 
We now extend $T$ to $E$ by setting $Tf=Tf^+-Tf^-$. This $T$ is the maximal extension of $\E[\cdot|\Sigma]$ as a conditional expectation operator, and we will denote it again by $\E[\cdot|\Sigma]$.
The space $E$ has the a.e. equivalence class of the constant $1$ function as a weak order unit.
The space $E$ is a $T$-universally complete Riesz space  with  weak order unit ${\bf 1}$ and $T$ is a strictly positive Riesz space conditional expectation operator on $E$ having $T{\bf 1}={\bf 1}$.  
If we take $\tau:\Omega\to\Omega$ be a map with $\tau^{-1}(A)\in{\cal A}$ and $\E[\chi_{\tau^{-1}(A)}|\Sigma]=\E[\chi_{A}|\Sigma]$, for all $A\in {\cal A}$ and set 
$Sf:=f\circ \tau$, the Koopman map, then $S$ is a Riesz homomorphism on $E$ with $S{\bf 1}={\bf 1}$ and $TS=T$.  
Further if for each $A\in {\cal A}$ there is $B_A\in {\cal A}$ so that $\mu(A\Delta \tau^{-1}(B_A))=0$ then $S$  is a surjective.
 
The system $(E,T,S,e)$ is a conditional expectation preserving system, with $S$ surjective 
and
$$L_Sf=\lim_{n\to\infty}\frac{1}{n}\sum_{k=0}^{n-1}f\circ \tau^k$$
converges a.e. pointwise to a conditional expectation operator on $E$ (which when restricted to 
$L^1(\Omega,{\cal A},\mu)$ is a classical conditional expectation operator).  The system $(E,T,S,e)$ is conditionally ergodic if and only if $L_S=T$.
 
 Then Theorem \ref{kakutani1} gives that if $n \in \N$ and $A\in {\cal A}$,
then there exists  $B\in {\cal A}$ with $\E[A|\Sigma]>0$ a.e. on $B$ such that 
$B,  \tau^{-1}(B), \dots, \tau^{1-n}(B)$ are a.e.  pairwise disjoint and 
\begin{equation}\label{april-thm-2025}
\E\left[\left.\bigcup_{j=0}^{n-1} \tau^{-j}B\right|\Sigma\right] \ge \left(\chi_{\{\omega|\E[\chi_A|\Sigma](\omega)>0\}} - (n-1) \E[A|\Sigma]\right)^+.
\end{equation}
\end{example}

\section{Aperiodicity and an $\epsilon$-bounded decomposition}

A probability space $(\Omega,{\cal B}, \mu)$, $\mu$ is nonatomic if for any $A \in {\cal B}$ with $\mu(A)>0$ there exists $B \in {\cal B}$ with $B \subset A$ and $0< \mu(B) < \mu(A)$. 
On nonatomic measure spaces,  an $\epsilon$-bounded version of the Kakutani-Rokhlin decomposition can be obtained, see \cite[Corollary 6.25]{eisner2015operator} and \cite[Lemma 4.7]{petersen1989ergodic}:

\begin{theorem}
Let $\tau:\Omega \rightarrow \Omega$ be an ergodic measure preserving transformation on a nonatomic measure space $(\Omega, {\cal B}, \mu)$, $n\in\Z$ and $\epsilon >0$, then there is a measurable set $B \subset \Omega$ such that $B, \; \tau^{-1} B, \;\dots\; \tau^{1-n}B$\ are pairwise disjoint and cover $\Omega$ up to a set of measure less than $\epsilon$. 
\end{theorem}

The original $\epsilon$-bounded version of the decomposition, as developed by Rokhlin, see \cite[page 10]{rokhlin1967lectures},  was posed for ergodic measure preserving systems which are  aperiodic. See also \cite{kakutani}.

On a probability space $(\Omega,{\cal B}, \mu)$,  an aperiodic transformation is a transformation whose periodic points form a set of measure $0$ (see \cite{rokhlin1967lectures}), that is $\mu(\{x \in \Omega \; | \;  \tau^p x=x \mbox{ for some } p \in \N \})=0$.
We recall Rokhlin's 1943 version of the Kakutani-Rokhlin Lemma requiring aperiodicity,  quoted from \cite{weiss1989work}. 	

\begin{theorem}
	If $\tau$ is an aperiodic automorphism, then for any natural number $n$ and any positive $\epsilon$, there exists a measurable set $A \subset \Omega$ such that the sets $A, \tau^{-1} A,...,\tau^{1-n}A$ are pairwise disjoint and the complement of their union has measure less than $\epsilon$.	
\end{theorem}	

Let $(\Omega,{\cal B}, \mu)$ be a probability space. 
The measure $\mu$ is said to be continuous if for any $A \in {\cal B}$ with $\mu(A)>0$ and any $\alpha \in \R$ with $0< \alpha < \mu(A)$ there exists $B \in {\cal B}$ with $B \subset A$ and $ \mu(B)=\alpha$.  Note that every continuous measure is nonatomic.
If $\mu$ is a continuous 
measure and $\tau$ is an ergodic measure preserving transformation, then $\tau$ is aperiodic. 
Indeed, as $\tau$ is ergodic, there exists $p \in \N$ such that $\mu(A_p)>0$, 
where $A_p=\{x \in \Omega \; | \;\tau^p x=x \}$. 
Choose $B \subset A_p$ with $0 < \mu(B) < \frac{1}{p}$, which is possible as $\mu$ is 
a continuous measure. 
The set $C=B \cup \tau^{-1}B \cup \dots \cup \tau^{1-p}B$ is $\tau$-invariant and satisfies $0< \mu(C) <1$ contradicting the assumption of ergodicity. 

An aperiodic transformation on a continuous measure space need not be ergodic.  
For example, consider the unit square $[0,1]\times[0,1]$ with Lebesgue measure.
The transformation 
$$\tau(x,y)=((x+\alpha)\mbox{ mod }1,y) \; \; \forall\, (x,y) \in [0,1]\times [0,1],$$
with $\alpha \in [0,1]$ irrational,  is aperiodic but not ergodic.

We extend these decompositions to the measure free context of Riesz spaces and begin by giving (a non-pointwise) definition of periodicity in the setting of Riesz spaces.

\begin{definition}
Let $(E,T,S,e)$ be a conditional expectation preserving system and $v$ be a component of $e$.
We say that $(S,v)$ is periodic if there is $N \in \N$ so that for all components $c$ of $v$ with  $0 \ne c \ne v$ we have that  $q(c,k)=0$ for all $k \ge N$.	
\end{definition}

The logic of this definition is that for all such $c$ we have
$$c=\bigvee\limits_{k=1}^{N-1} q(c,k)$$
and $S^kq(c,k)\le c$,
for $k=1,\dots, N-1.$ 

We note that, as in the measure theoretic setting, aperiodicity is defined as a stronger constraint than the negation of periodicity.

\begin{definition}
Let $(E,T,S,e)$ be a conditional expectation preserving system and $v\ne 0$ be a component of $e$.
We say that $(S,v)$ is aperiodic,  if for each $N\in\N$ and each component $c\ne 0$ of $v$,  there exists $k \ge N$ and a component $u$ of $c$ with $q(u,k) \ne 0$.  
\end{definition}

\begin{theorem}
Consider $E=L^1(\Omega,{\cal B},\mu)$ a probability space with $Tf:=\mathbb{E}[f]{\bf 1}$, where $e:={\bf 1}$ is the constant function with value $1$ a.e.,  and $Sf:=f\circ \tau$ is the von Neumann map generated by $\tau$,  a measure preserving transformation with $\tau$ a.e. surjective.  Further assume that there is $G\in {\cal B}$ with $0<\mu(G)<1$.  In this case the measure theoretic definition of aperiodicity of $\tau$ is equivalent to the Riesz spaces definition of $(S,e)$ being aperiodic. 
\end{theorem}

\begin{proof}
Suppose that $(S,e)$ is aperiodic, i.e. for each $N \in \N$ and each component $c$ of $e$ there exists $k \ge N$ and a component $u$ of $c$ with $q(u,k) \ne 0$. 
Let $A$ denote the set of periodic points of $\tau$.
By the way of contradiction, suppose that $\mu(A)>0$.
Hence there exists $N \in \N$ such that $\mu(A_N) >0$,
 where $A_N$ is the set of points of period $N$. 
Let $c:=\chi_{A_N}$  then from the aperiodicity of $(S,v)$ there is a component $u$ of $c$ and $k>N$ so that $q(u,k)\ne 0$.  Here there is a measurable subset $B$ of $A_N$ so that $u=\chi_B$.
Here all points of $B$ are of period $N$, giving $q(u,j)=0$ for all $j\ge N$, contradicting $q(u,k)\ne 0$.

Conversely, if the set of periodic points of $\tau$ has measure zero,  we show that $(S,e)$ is aperiodic.

Developing on  \cite[Lemma 3.12]{rudolph}, we give a meaning to the set of periodic points of $\tau$ having measure zero  in a point-less setting. Let  $p_k:=\chi_{A_k}$, where
$A_k$ is the a.e. maximal measurable set which has every measurable subset invariant under $\tau^{-k}$.  In the Riesz space terminology $p_k$ is the maximal component of $e$ with $S^kv=v$ for each $v$ a component of $p_k$.  Now $\tau$ being aperiodic gives that $p_k=0$ for all $k\in\N$. 

Suppose that $(S,e)$ is not aperiodic, then 
there exist $N\in\N$ and a component $c\ne 0$ of $e$, so that, for all $k \ge N$ and components $u$ of $c$ we have $q(u,k) = 0$.  Hence
$$u=\sum_{k=1}^{N-1}q(u,k)$$
 for all $k\ge N$ and $u$ a component of $c$.
 Thus
 $$S^{N!}u=u$$
 for all components $u$ of $c$,  making $c$ a component of $p_{N!}=0$.
 Thus $0< c\le p_{N!}=0$, a contradiction.  Thus $(S,e)$ is aperiodic. 
 \end{proof}

 \begin{lemma}\label{aperiodic-lemma}
Let	$(E, T, S, e)$ be a conditionally ergodic conditional expectation preserving system with $E$ $T$-universally complete and $(S,e)$ aperiodic, then, for each $N\in\N$, there is a component $c_N$ of $e$ with $P_{Tc_N}e=e$ and $S^ic_N\wedge S^jc_N=0$ for all $i,j=0,\dots,N$ with $i\ne j$.
 \end{lemma}
 
 \begin{proof}
 For each component $p\ne 0$ of $e$ let 
 $$K_N(p)=\sum_{k=N+1}^\infty q(p,k).$$
 Here $K_N(p)$ is a component of $p$ and, by Lemma \ref{q-disjoint}, 
 $$0=S^iK_N(p)\wedge S^jK_N(p)$$
 for all $0\le i< j\le N$. 

 Let 
 $${\frak G}:=\{(p,TK_N(p))\,|\,p \mbox{ a component of } e\}.$$
 Here $(e,0)\in {\frak G}$ so ${\frak G}$ is non-empty.
 We partially order ${\frak G}$ by
 $(p,TK_N(p))\le (\tilde{p},TK_N(\tilde{p}))$ if and only if 
 $p\le \tilde{p}$ and $TK_N(p)\le TK_N(\tilde{p})$.

 If $(p, Tk_N(p))_{p\in \Lambda}$ is a chain (totally ordered subset) in 
 ${\frak G}$, let 
 $$\hat{p}=\bigvee_{p\in \Lambda} p.$$
 Here $$\hat{p}=\lim_{p\in \Lambda} p$$
where $(p)_{p\in\Lambda}$  is an upwards directed net, directed by the partial ordering in the Riesz space.
By Lemma \ref{decomposition p}, 
$$K_N(p)=p-\sum_{k=1}^N q(p,k)$$
making $K_N(p)$ order continuous in $p$, see the definition of $q(p,k)$.
Thus 
$$TK_N(\hat{p})=\lim_{p\in\Lambda} TK_N(p).$$
Further, by the ordering on ${\frak G}$, the net $(TK_N(p))_{p\in\Lambda}$ is increasing and bounded, thus having
$$\lim_{p\in\Lambda} TK_N(p)=\bigvee_{p\in\Lambda} TK_N(p).$$
Hence we have
$$TK_N(\hat{p})=\bigvee_{p\in\Lambda} TK_N(p),$$
making $(\hat{p}, Tk_N(\hat{p}))$ an upper bound (in fact the supremum) for 
$(p, Tk_N(p))_{p\in \Lambda}$.

Thus Zorn's Lemma can be applied to 
  ${\frak G}$ to give that it has a maximal element, say 
$(p, Tk_N(p))$.

If $P_{TK_N(p)}e\ne e$, let $u=e-P_{TK_N(p)}e$, then $u$ is a non-zero component of $e$, so
 by the aperiodicity of $(S,e)$,  there exists $k>N$ and $c_k(u)$ a component of $u$ with $q(c_k(u),k)\ne 0$.  As $u\in R(T)$ we have $S^ju=u$, for all $j\in \Z$, and $p\le e-u\in R(T)$ which give $$K_N(c_k(u)+p)=K_N(c_k(u))+K_N(p)>K_N(p).$$
 Thus 
 $$(p,TK_N(p))<(p+c_k(u),TK_N(p+c_k(u))\in {\frak G},$$
 contradicting the maximality of  $(p,TK_N(p))$.  Hence $P_{TK_N(p)}=e$ and setting
 $c_N=K_N(p)$ concludes the proof.
 \end{proof}
 
The Kakutani-Rokhlin Lemma with $\epsilon$-bound can be formulated in a Riesz space as follows.

\begin{theorem}[Riesz space Kakutani-Rokhlin]\label{theo ergodic}
Let	$(E, T, S, e)$ be a conditionally ergodic conditional expectation preserving system with $E$ $T$-universally complete and $S$ surjective.  If $(S,e)$ is aperiodic, $n \in \N$ and $\epsilon >0$ then there exists a component $q$ of $e$ in $E$ with $(S^i q)_{i=0,...,n-1}$ disjoint and
\begin{equation}\label{kr-ep}
T\left(e-\bigvee\limits_{i=0}^{n-1}S^i q\right) \le \epsilon e.
\end{equation}
\end{theorem}

\begin{proof}
Let $n \in \N$ and $\epsilon >0$.  Take $N >\frac{n-1}{\epsilon}$.  By Lemma \ref{aperiodic-lemma}, there exists a component $c_N$ of $e$ with $P_{Tc_N}e=e$ and $S^ic_N\wedge S^jc_N=0$ for all $i,j=0,\dots,N$ with $i\ne j$. Let $p:=c_N$. Then $q(p,k)=0$ for all $k=1,\dots,N$ giving that
\begin{equation}\label{19-june}
Np \le n(p).
\end{equation} 
By Theorem \ref{kac-thm} we have
\begin{equation}\label{baw-12-1}
e=P_{Tp}e=Tn(p).
\end{equation}
Combining (\ref{19-june}) and (\ref{baw-12-1}), we get 
\begin{equation}\label{baw-12-2-1}
NTp \le Tn(p)=e.
\end{equation}
Since $N>\frac{n-1}{\epsilon}$, (\ref{baw-12-2-1}) yields
\begin{equation}\label{baw-12-2}
Tp \le \frac{\epsilon}{n-1}e.
\end{equation}

By Theorem \ref{kakutani1}, there exists a component  $q$ of $P_{Tp}e=e$ such that $q, Sq , ...,S^{n-1}q$ are pairwise disjoint and 
\begin{equation}\label{baw-12-3}
T\left(\bigvee\limits_{j=0}^{n-1}S^j q\right) \ge P_{Tp}e - (n-1) Tp \ge e- \epsilon e
\end{equation}
which gives the inequality of the theorem.
\end{proof}

 On reading the works of Rokhlin, it appears that the requirement of conditional ergodicity is redundant and only aperiodicity is needed.  As we know, that every CEPS is conditionally ergodic with respect to $L_S$. So in our case conditional ergodicity can be dispensed with, but we need to be careful to use the conditional expectation operator $L_S$  and work in the $L_S$-universal completion of $E$,  which we will denote by $\hat{E}$.

\begin{corollary} \label{cor aperiodic}
Let	$(E, T, S, e)$ be a conditional expectation preserving system with $E$ $T$-universally complete and $S$ surjective, then 	$(\hat{E}, L_S, S, e)$ is a conditionally ergodic conditional expectation preserving system. If $v$ is a component of $e$ in $R(L_S)$ with $(S,v)$ aperiodic, $n \in \N$ and $\epsilon >0$ then there exists a component $q$ of $v$ in $E$ with $(S^i q)_{i=0,...,n-1}$ disjoint and 
	$$L_S\left(v-\bigvee\limits_{i=0}^{n-1}S^i q\right) \le \epsilon v.$$
\end{corollary}

Theorem \ref{theo ergodic} is the specific case of  Corollary \ref{cor aperiodic} where  $(E, T, S, e)$ is a conditionally ergodic and $E$ is $T$-universally complete, as then $L_S=T$. 

\begin{example} Continuing on Example \ref{remark-2025-1}, 
 let $n \in \N$ and $A\in {\cal A}$ be so that $L_S\chi_A = \chi_A$.
 If $\epsilon>0$ and in addition $\tau$ is an a.e. aperiodic map on $A$,  then
 Corollary \ref{cor aperiodic} gives that
there exists  $B\in {\cal A}$ with $B\subset A$ such that 
$B,  \tau^{-1}(B), \dots, \tau^{1-n}(B)$ are a.e.  pairwise disjoint and 
\begin{equation}\label{april-thm-2025}
0\le L_S(\chi_A-\chi_C) \le \epsilon\chi_A,
\end{equation}
where
$C=\bigcup_{j=0}^{n-1} \tau^{-j}B$.
\end{example}

To highlight the need for aperiodicity, we now give an example of a conditionally ergodic preserving system $(E,T,S,e)$ which is $T$-universally complete and is neither periodic nor aperiodic and for which the $\epsilon$ approximation of Theorem  
\ref{theo ergodic} fails.
 
\begin{example} 
Let $E_n=\ell(n)$, the space of real finite sequences of length $n$ with componentwise ordering. On $E_n$ we introduce the conditional expectation 
$$T_n(f_n)(i)=\frac{1}{n}\sum_{j=1}^n f_n(j)\mathbf{1}_n,  \, f_n\in E_n,$$ 
where $\mathbf{1}_n(j)=1$ for all $j=1,\dots, n$.  
Further $\mathbf{1}_n$ is a weak order unit for $E_n$.  
On each $E_n$ we take $S_n$ as the Riesz homomorphism given by  $S_1(f_1)=f_1$, and for $n\ge 2$, 
$$S_nf_n(j)=\left\{\begin{array}{ll}f_n((j-1)), &j=2,\dots,n\\  f_n(n),&j=1\end{array}\right..$$
Clearly $(E_n,T_n,S_n,\mathbf{1}_n)$ is a $T_n$-universally complete ergodic conditional expectation preserving system.
We take $(E,T,S,\mathbf{1})$ as the direct product of the spaces $(E_n,T_n,S_n,\mathbf{1}_n), n\in\N$. Now $$E=\prod_{n=1}^\infty E_n, \quad T=\prod_{n=1}^\infty T_n,\quad S=\prod_{n=1}^\infty S_n, \quad \mathbf{1}=(\mathbf{1}_1,\mathbf{1}_2,\dots). $$
The resulting space $(E, T, S, e)$ is a conditionally ergodic conditional expectation preserving system with $E$ $T$-universally complete.  Further $(S,e)$ is neither periodic nor aperiodic.
If we take $n\in\N\setminus\{1\}$ and $0<\epsilon<1/n$ in Theorem \ref{theo ergodic} and assume that there exists component $p$ of $\mathbf{1}$ exhibiting the required properties of the theorem, then the disjointness of $S^jp$ for $j=0,\dots,n$ gives that the components $p_j=0$ for $j\le n$, but then $S_j^kp_j=0$ for all $k$ and so (\ref{kr-ep}) fails.
\end{example}


 \section{Approximation of aperiodic maps}

 Rokhlin proved an interesting consequence of his lemma stating that any aperiodic transformation $\tau$ can be approximated by periodic ones, see \cite[page 75]{halmos2017lectures} and \cite{weiss1989work}. 
 That is, for any positive integer $n$ and any $\epsilon >0$, there exists a periodic transformation $\tau'$ of period $n$ such that $d(\tau, \tau' ) \le \frac{1}{n}+\epsilon$ where $d(\tau, \tau' )=\mu \{x : \tau x \ne \tau' x\}$.

 In this section we apply Theorem \ref{theo ergodic}  to obtain an approximation of aperiodic conditional expectation preserving transformations by periodic ones in the conditional Riesz space setting.
 
  \begin{theorem}
 Let $(E,T, S, e)$ be a $T$-universally complete  conditionally ergodic preserving system where $S$ is a surjective Riesz homomorphism and $(S,e)$ aperiodic. 
For each $1>\epsilon >0$, there  exists a surjective Riesz homomorphism $S'$ such that  $(E,T, S', e)$ is a conditional expectation preserving system with $(S',e)$ periodic  and  
 \begin{equation}\label{eg-bound}
 \sup_{u\in C_e} T |(S-S')u| \le \epsilon e.
 \end{equation}
 \end{theorem} 
 
\begin{proof}
Let $\epsilon >0$ and $n>4/\epsilon$.
By Theorem \ref{theo ergodic} there is a component $p$ of $e$ such that $(S^i p)_{i=0,...,n-1}$ are disjoint and 
 \begin{equation}\label{eg-1}
 T\left(e-h\right) \le \frac{\epsilon}{4} e
 \end{equation}
where
 \begin{equation}\label{eg-2}
h=\sum_{i=0}^{n-1} S^ip=\bigvee_{i=0}^{n-1} S^ip\in C_e.
\end{equation}
Hence 
 \begin{equation}\label{eg-12-2}
\left(1-\frac{\epsilon}{4}\right)e \le Th.
 \end{equation}

Further, applying $T$ to (\ref{eg-2}) gives
 \begin{equation}\label{eg-3}
nTp=Th\le Te=e,
\end{equation}
 so 
 \begin{equation}\label{eg-4}
Tp\le \frac{1}{n}e\le\frac{\epsilon}{4}e.
\end{equation}

We now give a decomposition of $e$ which will be the basis for the decomposition of $E$ into bands. 
 Let 
\begin{equation}\label{eg-12}
q=\sum_{i=0}^{n-2} S^ip=\bigvee_{i=0}^{n-2} S^ip\in C_e.
\end{equation}
 Here 
\begin{equation}\label{app-1}
h=S^{n-1}p+q=(S^{n-1}p)\vee q.
\end{equation} 
Hence we have the following disjoint decomposition of $e$ by its components $q, S^{n-1}p, e-h$, 
\begin{equation}\label{app-3}
q+({S^{n-1}p})+(e-h)=e.
\end{equation}

We define the approximation Riesz to $S$ as 
\begin{equation}\label{eg-app}
S'=SP_{q}+S^{1-n}P_{S^{n-1}p}+P_{e-h}.
\end{equation}
$S'$ is a finite sum of order continuous maps and is thus order continuous.

We begin by verifying that $(E,T,S',e)$ is a conditional expectation preserving system with $S'$ surjective.
 $S'$ is a sum of compositions of Riesz homomorphisms and is thus a Riesz homomorphism. 
From (\ref{app-3})  and (\ref{eg-app}) we get
$$S'e=Sq+p+(e-h)=e.$$

From (\ref{eg-app}) and (\ref{app-3}), as $TS=T$,
$$TS'= TSP_{q}+TS^{1-n}P_{S^{n-1}p}+TP_{e-h}=T(P_{e-h+q+S^{n-1}p})=T.$$
 As $T$ is strictly positive, the condition $TS'=T$ ensures that $S'$ is injective.  
 
 We now prove that 
 $S'$ is surjective.
 In particular, for $f\in E$ set
 $$\hat{f}=P_{e-h}f+P_{q}S^{-1}f+P_{S^{n-1}p}S^{n-1}f.$$
 We show that $S'\hat{f}=f$.
 To see this 
 \begin{eqnarray*}
 S'\hat{f}&=&( SP_{q}+S^{1-n}P_{S^{n-1}p}+P_{e-h})
 (P_{e-h}f+P_{q}S^{-1}f+P_{S^{n-1}p}S^{n-1}f)\\
 &=&P_{e-h}f+SP_{q}S^{-1}f+S^{1-n}P_{S^{n-1}p}S^{n-1}f.
 \end{eqnarray*}
 Here $$SP_qS^{-1}f=SP_{S^{-1}Sq}S^{-1}f=SS^{-1}P_{Sq}f=P_{Sq}f$$ 
 and
 $$S^{1-n}P_{S^{n-1}p}S^{n-1}f=S^{1-n}S^{n-1}P_{p}f=P_pf.$$
 Thus, by (\ref{eg-app}),
 \begin{eqnarray*}
 S'\hat{f}&=&P_{e-h}f+P_{Sq}f+P_{p}f=P_{(e-h)+Sq+p}f
 \end{eqnarray*}
 giving
 $S'\hat{f}=f$, showing that $S'$ is surjective.

We now show that $(S',e)$ is periodic.
  It suffices to show that, for each $u\in C_e$ with $u\ne 0$,  we have that 
$$\bigvee_{k=1}^{n} (S')^k u \ge u.$$
Let  $u\in C_e$ with $u\ne 0$.
Here
\begin{equation}\label{periodic-12-w-h}
S'u\ge S'(u\wedge (e-h))=u\wedge (e-h). 
\end{equation}

Let $0\le j\le n-1$ and $0\le i\le n-1-j.$  
We show inductively with respect to $i$ that
\begin{equation}\label{periodic-12-ind-1}
(S')^i(u\wedge S^jp)=S^i(u\wedge S^jp)\le S^{i+j}p\le q.
\end{equation}
For $i=0$,
\begin{equation}\label{periodic-12-ind-2}
(S')^0(u\wedge S^jp)=u\wedge S^jp=S^0(u\wedge S^jp)\le S^jp\le q.
\end{equation}
Now suppose
$0\le i\le n-1-j-1$ and that
\begin{equation}\label{periodic-12-ind-3}
(S')^i(u\wedge S^jp)=S^i(u\wedge S^jp)\le S^{i+j}p\le q.
\end{equation}
then applying $S'$ to (\ref{periodic-12-ind-3}),  from the definition of $S'$, we get
\begin{equation*}
S'(S')^{i}(u\wedge S^jp)=S(S')^{i}(u\wedge S^jp)=S^{i+1}(u\wedge S^jp)\le S^{i+j+1}p\le q,
\end{equation*}
from which (\ref{periodic-12-ind-1}) holds by induction.

In particular, for $i=n-j-1$,
\begin{equation}\label{periodic-12-app-S'-n-1}
(S')^{n-1-j}(u\wedge S^jp)=S^{n-1-j}(u\wedge S^jp)\le S^{n-1}p.
\end{equation}
Now applying $S'$ to the above gives
\begin{equation}\label{periodic-12-app-S'-n}
(S')^{n-j}(u\wedge S^jp)=S^{1-n}S^{n-1-j}(u\wedge S^jp)=S^{-j}(u\wedge S^jp)\le p=S^0p.
\end{equation}

Hence (\ref{periodic-12-app-S'-n}) can be written as
\begin{equation}\label{periodic-12-app-S'-*}
(S')^{n-j}(u\wedge S^jp)=(S^{-j}u)\wedge S^0p.
\end{equation}
So by (\ref{periodic-12-app-S'-*}) 
we have
$$(S')^n(u\wedge S^jp)=(S')^j(S')^{n-j}(u\wedge S^jp)=(S')^j((S^{-j}u)\wedge S^0p),$$
and now applying (\ref{periodic-12-ind-1}) (with
replacing $j$ by $0$, $i$ by $j$ and $u$ by $S^{-j}u$ ), we have
\begin{equation}\label{periodic-12-app-+}
(S')^n(u\wedge S^jp)=(S')^j((S^{-j}u)\wedge S^0p)=S^j((S^{-j}u)\wedge S^0p)=u\wedge S^jp.
\end{equation}
Taking the supremum of (\ref{periodic-12-app-+}) over $j=0,\dots,n-1$ gives
\begin{equation}\label{periodic-12-app-sup}
(S')^nu\ge (S')^n(u\wedge h)=\bigvee_{j=0}^{n-1}(S')^n(u\wedge S^jp)=u\wedge \left(\bigvee_{j=0}^{n-1}S^jp\right)=u\wedge h.
\end{equation}
Combining (\ref{periodic-12-w-h}) and (\ref{periodic-12-app-sup}) gives
$$(S')^nu \vee S'u\ge (u\wedge h)\vee (u\wedge (e-h))=u\wedge e=u,$$
showing that $(S',e)$ is periodic.

Finally
we  show that $S'$ obeys the bound (\ref{eg-bound}).
 For $u\in {\cal C}_e$ we have
 $$(S-S')u=(S-S^{1-n})P_{S^{n-1}p}u+(S-I)P_{e-h}u.$$
Here, as $TS^j=T, j\ge 0,$ by (\ref{eg-4}),
$$T|(S-S^{1-n})P_{S^{n-1}p}u|\le TS^np+Tp= 2Tp\le \frac{\epsilon}{2} e$$
 and, by (\ref{eg-1}),
 $$T|(S-I)P_{e-h}u|\le TS(e-h)+T(e-h)=2T(e-h)\le \frac{\epsilon}{2}e,$$
 giving
 $$T |(S-S')u| \le T|(S-S^{1-n})P_{S^{n-1}p}u|+T|(S-I)P_{e-h}u|\le  \epsilon e.$$
Thus (\ref{eg-bound}) holds.
 \end{proof}


\begin{Backmatter}

\begin{ack}
This research was funded in part by the joint South Africa - Tunisia Grant (South African National Research Foundation Grant Number SATN180717350298, grant number 120112. 
\\

\noindent
Competing interests: The authors declare none.
\end{ack}

\end{Backmatter}

\end{document}